\documentclass[12pt, reqno]{amsart}
\usepackage{graphicx}
\usepackage{amsmath, amsthm, amscd, amsfonts, amssymb, graphicx, color}
\usepackage[bookmarksnumbered, colorlinks, plainpages]{hyperref}
\input{mathrsfs.sty}

\newtheorem{theorem}{Theorem}[section]
\newtheorem{lemma}[theorem]{Lemma}
\newtheorem{proposition}[theorem]{Proposition}
\newtheorem{corollary}[theorem]{Corollary}
\theoremstyle{definition}

\theoremstyle{remark}
\newtheorem{remark}[theorem]{Remark}
\numberwithin{equation}{section}
\begin{document}
\title{An extension of the L\"owner--Heinz inequality}

\author[M.S. Moslehian, H. Najafi]{Mohammad Sal Moslehian and Hamed Najafi}

\address{Department of Pure Mathematics, Center of Excellence in Analysis on Algebraic Structures (CEAAS), Ferdowsi University of
Mashhad, P.O. Box 1159, Mashhad 91775, Iran}
 \email{moslehian@ferdowsi.um.ac.ir
and moslehian@member.ams.org}
\email{hamednajafi20@gmail.com}



\begin{abstract}
We extend the celebrated L\"owner--Heinz inequality by showing that
if $A, B$ are Hilbert space operators such
that $A > B \geq 0$, then
\begin{eqnarray*}
A^r - B^r \geq ||A||^r-\left(||A||-
\frac{1}{||(A-B)^{-1}||}\right)^r > 0
\end{eqnarray*}
for each $0 < r \leq 1$. As an application we prove that
\begin{eqnarray*}
\log A - \log B \geq \log||A||-
\log\left(||A||-\frac{1}{||(A-B)^{-1}||}\right)>0.
\end{eqnarray*}
\end{abstract}

\maketitle
%-------------------------------------------------------------------------------

\section{Introduction}
Let $(\mathscr{H}, \langle \cdot,\cdot\rangle)$ be a complex Hilbert
space and $\mathbb{B}(\mathscr{H})$ denote the algebra of all
bounded linear operators on $\mathscr{H}$ equipped with the operator
norm $\|\cdot\|$. There are three types of ordering on the real
space of all self-adjoint operators as follows. Let $A,B \in
\mathbb{B}(\mathscr{H})$ be self-adjoint. Then\\
(1) $A \geq B$ if $\langle Ax, x\rangle \geq \langle Bx,
x\rangle $.\\
(2) $A \succ B$ if $\langle Ax, x\rangle > \langle Bx, x\rangle $
holds for all non-zero elements $x\in \mathscr{H}$. \\
(3) $A > B$ if $A\geq B$ and $A-B$ is invertible.\\
Clearly (3)$\Rightarrow$ (2)$\Rightarrow$ (1) but the reverse
implications are not valid in general. For instance, if $A$ is the
diagonal operator $(1, 1/2, 1/3, \cdots)$ on $\ell^2$, then $A\succ
0$ but $A\ngtr 0$. Of course, in the case where $H$ is of finite
dimension, (2) and (3) are equivalent. A continuous real valued
function $f$ defined on an interval $J$ is called operator monotone
if $A\geq B$ implies that $f(A)\geq f(B)$ for all self-adjoint
operators $A, B $ with spectra in $J$. The L\"owner--Heinz
inequality says that, $f(x)= x^r\,\,(0 < r \leq 1)$ is operator
monotone on $[0, \infty)$. L\"owner \cite{LOW} proved the inequality
for matrices. Heinz \cite{HEI} proved it for positive operators
acting on a Hilbert space of arbitrary dimension. Based on the
$C^*$-algebra theory, Pedersen \cite{PED} gave a shorter proof of
the inequality.

There exist several operator norm inequalities each of which is
equivalent to the L\"owner-Heinz inequality, see \cite{FUR}. One of
them is $\|A^rB^r\|\le\|AB\|^r$, called the C\"ordes inequality in
the literature, in which $A$ and $B$ are positive operators and $0 <
r\le 1$. A generalization of the C\"ordes inequality for operator
monotone functions is given in \cite{FF}. It is shown in \cite{ACS}
that this norm inequality is related to the Finsler structure of the
space of positive invertible elements.

Kwong \cite{KWO} showed that if $A > B$ ($A \succ B$, resp.), then
$A^r> B^r$ ($A^r \succ B^r$, ressp.) for $0<r\leq 1$. Uchiyama
\cite{Uc} showed that for every non-constant operator monotone
function $f$ on an interval $J$, $A \succ B$ implies $f(A) \succ
f(B)$ for all self-adjoint operators $A, B$ with spectra in $J$.

There are several extensions of the L\"owner--Heinz inequality. The
Furuta inequality \cite{FPAMS}, which states that if $A \ge B \ge
0$, then for $r\ge 0$, $(A^{r/2}A^p A^{r/2}) ^{1/q}\ge(A^{r/2} B^p
A^{r/2})^{1/q}$ holds for $p \ge 0$ and $q \ge 1$ with $(1+r)q \ge
p+r$, is known as an exquisite extension of the L\"owner--Heinz
inequality; see the survey article \cite{FUJ} and references
therein.

If $f$ is an operator monotone function on $(-1,1)$, then $f$ can be
represented as
\begin{eqnarray}\label{7-3}
f(t)=f(0)+f^{'}(0)\int_{-1}^{1}\frac{t}{1-\lambda t} d\mu(\lambda)
\end{eqnarray}
where $\mu$ is a positive measure on $(-1,1)$. It is known that
\begin{eqnarray}\label{0-1}
t^r = \frac{\sin(r\pi)}{\pi} \int_0^{\infty} \frac{t }{\lambda+t}
\lambda^{r-1} d\lambda,
\end{eqnarray}
in which $0<r<1$, and
\begin{eqnarray}\label{0-2}
A^r = \frac{\sin(r\pi)}{\pi} \int_0^{\infty} \frac{A }{\lambda+A}
\lambda^{r-1} d\lambda,
\end{eqnarray}
where $A$ is positive and $0<r<1$; see e.g. \cite[Chapter V]{BH}.

In this paper we extend the L\"owner--Heinz inequality by showing
that if $A, B \in \mathbb{B}(\mathscr{H})$ such that $A > B \geq 0$,
then
\begin{eqnarray*}
A^r - B^r \geq ||A||^r-\left(||A||-
\frac{1}{||(A-B)^{-1}||}\right)^r > 0
\end{eqnarray*}
for each $0 < r \leq 1$. As an application we prove that
\begin{eqnarray*}
\log A - \log B \geq \log||A||-
\log\left(||A||-\frac{1}{||(A-B)^{-1}||}\right)>0.
\end{eqnarray*}

\section{The results}

We start our work with the following useful lemma.

\begin{lemma}\label{11}
Let $A,B \in \mathbb{B}(\mathscr{H})$ be invertible positive
operators such that $A - B \geq m >0$. Then
\begin{eqnarray}\label{2-1}
 B^{-1} - A^{-1} \geq \frac{m}{(||A||-m) \ ||A||}\,.
\end{eqnarray}
\end{lemma}

\begin{proof}
Since $f(t)=\frac{1}{t}$ is a decreasing operator monotone function
on $[0,\infty)$ we have $B^{-1} \geq (A - m)^{-1}$. On the other
hand
\begin{eqnarray*}
&&(A -m)^{-1} \geq A^{-1} + \frac{m}{(||A||-m)  ||A||} \\
&\Longleftrightarrow& (A^{-1} + \frac{m}{(||A||-m)  ||A||})(A-m) \leq 1 \\
&\Longleftrightarrow&   \frac{A^2}{(||A||-m)||A||} - \frac{mA}{(||A||-m)||A||} \leq 1\\
&\Longleftrightarrow&  A^2 - m A \leq (||A||-m)||A||\\
&\Longleftrightarrow&  ||A^2 - m A|| \leq (||A||-m)||A||.
\end{eqnarray*}
There exists $\lambda_0\in {\rm sp}(A)$ such that $||A||=\lambda_0$.
Since $A \geq m >0$, we have
\begin{eqnarray*}
||A^2 - m A||&=& \max \{\lambda \ : \ \lambda\in {\rm sp}(A^2-mA)\} \\
&=& \max \{\lambda^2-m \lambda \ : \ \lambda\in {\rm sp}(A)\}\\
&=& \lambda_0^2-m \lambda_0 \\
&=& (||A||-m)||A||.
\end{eqnarray*}
So $B^{-1} \geq (A - m)^{-1}  \geq A^{-1} + \frac{m}{(||A||-m)
||A||}$.
\end{proof}
%====================================================================================================

Now we use Lemma \ref{11} to prove an analogous but different result
to the main theorem of Uchiyama \cite{Uc} in an easy fashion as an
offshoot of our work.

\begin{proposition}
Let $f$ be a non-constant operator monotone function on an interval
$J$ and $A, B $ be self-adjoint operators with spectra in $J$ such
that $A >B$. Then $f(A)>f(B)$.
\end{proposition}
\begin{proof}
 Without loss of generality we assume that $J=(-1,1)$. Let $A,B
\in \mathbb{B}(\mathscr{H})$ be self-adjoint operators with spectra
in $(-1,1)$ and $A-B$ is positive and invertible. So there exists $m>0$ such that $A-B \geq m
>0$. Put $f_{\lambda}(t)=\frac{t}{1-\lambda t}$ for each $\lambda$ with $|\lambda|< 1$.
We shall show that $f_{\lambda}(A)-f_{\lambda}(B)$ is bounded blow and so invertible.
 It is clear that the claim is true for $\lambda = 0$. If $0 <\lambda<
1$, then $(1-\lambda B)- (1-\lambda A) =\lambda (A-B)>  \lambda m
>0$ as well as $1-\lambda B$ and $1-\lambda A$ are positive invertible operators. Since
\begin{eqnarray*}\label{8-2}
\frac{t}{1-\lambda
t}=\frac{-1}{\lambda}+\frac{1}{\lambda}\left(\frac{1}{1-\lambda
t}\right),
\end{eqnarray*}
by Lemma \ref{11}, we have
\begin{eqnarray*}
f_{\lambda}(A)-f_{\lambda}(B)&=&\frac{1}{\lambda}\left(\frac{1}{1-\lambda A}-\frac{1}{1-\lambda B}\right)\\
&\geq& \frac{1}{\lambda}\left( \frac{\lambda m}{(||1-\lambda
B||-\lambda m) \ ||1-\lambda B||} \right)
 \ \ \ \ \ \ \
(\rm{by~} \eqref{11}) \\
&=&\frac{m}{(||1-\lambda B||-\lambda m) \ ||1-\lambda B||} > 0
\end{eqnarray*}
A similar argument shows that
 \begin{eqnarray*}
 f_{\lambda}(A)-f_{\lambda}(B) \geq
 \frac{m}{(||1-\lambda A||+\lambda m) \ ||1-\lambda A||} > 0
 \end{eqnarray*}
 for each $-1 <\lambda< 0$. Since $f$
 is operator monotone on $(-1,1)$, it can be represented as
\begin{eqnarray*}
f(t)=f(0)+f^{'}(0)\int_{-1}^{1}f_{\lambda}(t) d\mu(\lambda)\,,
\end{eqnarray*}
where $\mu$ is a nonzero positive measure on $(-1,1)$. Since $f$ is
nonconstant, $f^{'}(0)>0$, \cite[Lemma  2.3]{B-S}. Hence
\begin{eqnarray*}
f(A)&-&f(B)\\
&=&f^{'}(0)\int_{-1}^{1}\left(\frac{A}{1-\lambda A} -  \frac{B}{1-\lambda B}\right) d\mu(\lambda)\\
&=&f^{'}(0)\int_{-1}^{1}\left(f_{\lambda}(A) -  f_{\lambda}(B)\right) d\mu(\lambda)\\
&\geq& f^{'}(0)\int_{-1}^{1} m_{\lambda} d\mu(\lambda),
\end{eqnarray*}
where
$$m_{\lambda}= \frac{m}{(||1-\lambda B||-\lambda m) \ ||1-\lambda B||}$$ if $0\leq \lambda
<1$, and $$m_{\lambda}= \frac{m}{(||1-\lambda A||+\lambda m) \
||1-\lambda A||}$$ if $-1<\lambda<0$. Since $\mu$ is a nonzero
positive measure and $m_{\lambda}>0$, we have
 \begin{eqnarray*}
f(A)-f(B)\geq f^{'}(0)\int_{-1}^{1} m_{\lambda} d\mu(\lambda) > 0.
\end{eqnarray*}
Therefore $f(A)>f(B)$.
\end{proof}

%===================================================================================================

Our main result reads as follows.

\begin{theorem}\label{1-1}
Let $A,B \in \mathbb{B}(\mathscr{H})$ be positive operators such
that $A - B \geq m > 0 $ and $0<r \leq 1$. Then
$$ A^{r} - B^{r} \geq ||A||^r-(||A||-m)^r\,.$$
\end{theorem}
\begin{proof}
Let $0 < r < 1$. First note that,
\begin{eqnarray*}
\frac{A}{\lambda+A}- \frac{B}{\lambda+B}
&=& \lambda\left(\frac{1}{\lambda + B}-\frac{1}{\lambda+ A}\right)\\
&\geq& \frac{\lambda m}{(||A+\lambda||-m)||A+\lambda||}  \ \ \rm{by} \ (\ref{2-1}) \\
&=& \frac{\lambda m}{(||A||+\lambda -m)(||A||+\lambda)}
\end{eqnarray*}
for each $\lambda > 0$. By using (\ref{0-2})  we have
\begin{align*}
A^{r}&-B^{r}\\
&= \frac{\sin(r\pi)}{\pi} \int_{0}^{\infty} \lambda^{r-1}\left( \frac{A}{\lambda+A}- \frac{B}{\lambda+B}\right) d\lambda\\
& \geq  \frac{\sin(r\pi)}{\pi} \int_{0}^{\infty} ( \frac{m
\lambda^{r} }{(||A||+\lambda -m)(||A||+\lambda)})d\lambda,
\end{align*}
We need to compute
\begin{eqnarray*}
I=\int_0^{\infty}
\frac{\lambda^{r}}{(\lambda+||A||)(\lambda+(||A||-m))} d\lambda
\end{eqnarray*}
where $0<r<1$. We will need the branch cut for $z^r=\rho^r
e^{ir\theta}$, in which $z=\rho e^{i\theta}$ and $0\leq \theta \leq
2\pi$. Consider
\begin{eqnarray*}
\int_{C} \frac{z^{r}}{(z+||A||)(z+(||A||-m))} dz\,,
\end{eqnarray*}
where the keyhole contour $C$ consists of a large circle $C_R$ of
radius $R$, a small circle $C_{\epsilon}$ of radius $\epsilon$  and
two lines just above and below the branch cuts $\theta =0$; see
Figure 1. The contribution from $C_{R}$ is $O(R^{r-2} ) 2\pi R =
O(R^{r-1})=0$ as $R\rightarrow \infty$. Similarly the contribution
from $C_{\epsilon}$ is zero as $\epsilon\rightarrow 0$. The
contribution from just above the branch cut and from just below the
branch cut is $I$ and $-e^{2r\pi i}I$, respectively, as
$\epsilon\rightarrow 0$ and $R\rightarrow \infty$. Hence, taking the
limits as $\epsilon\rightarrow 0$ and $R\rightarrow \infty$,
\begin{figure}[her!]
\centering
\includegraphics[width=0.4\textwidth]{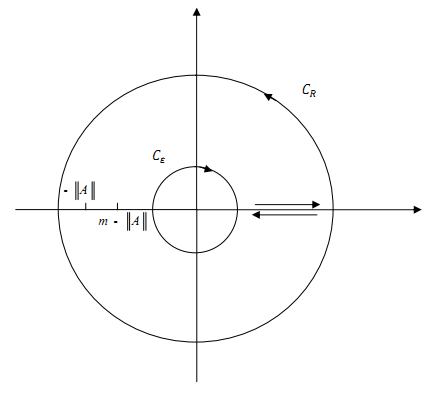}\\
\caption{Keyhole contour} \label{5-5}
\end{figure}
 \begin{eqnarray*}
 (1-e^{2r\pi i})I&=&\int_{C} \frac{z^{r}}{(z+||A||)(z+(||A||-m))} dz\\
 &=& -2 \pi i e^{r\pi i}\left(\frac{||A||^r-(||A||-m)^r}{||A||-(||A||-m)}\right)
\end{eqnarray*}
by the Cauchy residue theorem. So
\begin{eqnarray*}
 I=\frac{\pi}{m\sin(r\pi)}\left(||A||^r-(||A||-m)^r\right)\,.
\end{eqnarray*}
Therefore
\begin{align*}
A^{r}&-B^{r}\\
& \geq \frac{\sin(r\pi)}{\pi} \int_{0}^{\infty} ( \frac{m \lambda^{r}}{(||A||+\lambda -m)(||A||+\lambda)})d\lambda \\
&= ||A||^r-(||A||-m)^r.
\end{align*}
\end{proof}

%======================================================================================%
\begin{corollary}\label{ccc}
Let $A,B \in \mathbb{B}(\mathscr{H})$ be positive operators such
that $A - B \geq m >0 $. Then
$$ \log A - \log B \geq \log||A||- \log(||A||-m)\,.$$
\end{corollary}
\begin{proof}
Put $f_n(t)=n(t^{1 \over n}-1)$ on $[0,\infty)$. Then the sequence
$\{f_n\}$ uniformly converges to $\log t$ on any compact subset of
$(0,\infty)$. Hence
\begin{eqnarray*}
 \log A - \log B &=& \lim_{n\rightarrow \infty} f_n(A) - f_n(B)\\
 &\geq& \lim_{n\rightarrow \infty} n(||A||^{1\over n}-(||A||-m)^{1 \over n})  \\
 &=& \log||A||- \log(||A||-m).
 \end{eqnarray*}
\end{proof}
%==========================================================================================%

\begin{corollary}\label{12}
 Let $A,B \in \mathbb{B}(\mathscr{H})$ such that $A > B\geq 0$. Then
\begin{eqnarray*}
{\rm (i)} A^r &-& B^r\\
&\geq& ||A||^r-\left(||A||- \frac{1}{||(A-B)^{-1}||}\right)^r
\end{eqnarray*}
for all $0 < r \leq 1$
\begin{eqnarray*}
{\rm (ii)} \log A &-& \log B \\
&\geq& \log||A||- \log\left(||A||-\frac{1}{||(A-B)^{-1}||}\right)\,.
\end{eqnarray*}
\end{corollary}
\begin{proof}
It follows from $A > B\geq 0$ that $A -B \geq \frac{1}{||(A-B)^{-1}||}>0$. Now the assertions are deduced from Theorem \ref{1-1} and Corollary \ref{ccc}.
\end{proof}
%======================================================================================%
\begin{remark}
The inequality in Corollary \ref{12} is sharp. Indeed for positive
scalars $a,b$, if $a>b$, then
$$a^r - b^r = a^r-\left(a- \frac{1}{(a-b)^{-1}}\right)^r$$
and
\begin{align*}
\log a &- \log b \\
&=  \log a- \log\left(a-\frac{1}{(a-b)^{-1}}\right)\,.
\end{align*}

\end{remark}

\bigskip
%========================================================================================
\bibliographystyle{amsplain}

\end{document}